\documentclass[a4paper,twoside,12pt]{amsart}
\RequirePackage{amsmath,amssymb,xypic,latexsym}
\usepackage[all]{xy}
\newcommand{\Ac}{\mathcal{A}}
\newcommand{\Bc}{\mathcal{B}}
\newcommand{\Cc}{\mathcal{C}}
\newcommand{\Fc}{\mathcal{F}}
\newcommand{\Gc}{\mathcal{G}}
\newcommand{\Hc}{\mathcal{H}}
\newcommand{\Kc}{\mathcal{K}}
\newcommand{\Ll}{\mathcal{L}}
\newcommand{\Pc}{\mathcal{P}}
\newcommand{\CC}{\mathbb{C}}

\newcommand{\RR}{\mathbb{R}}

\theoremstyle{plain}
\newtheorem{theorem}{Theorem}[section]

\newtheorem{prop}[theorem]{Proposition}
\newtheorem{lem}[theorem]{Lemma}

\theoremstyle{definition}
\newtheorem{Def}[theorem]{Definition}

\newtheorem{remark}[theorem]{Remark}

\numberwithin{equation}{section}

\newcommand{\eq}[1][r]
   {\ar@<-3pt>@{-}[#1]
    \ar@<-1pt>@{}[#1]|<{}="gauche"
    \ar@<+0pt>@{}[#1]|-{}="milieu"
    \ar@<+1pt>@{}[#1]|>{}="droite"
    \ar@/^2pt/@{-}"gauche";"milieu"
    \ar@/_2pt/@{-}"milieu";"droite"}

\renewcommand{\emptyset}{\varnothing}

\newcommand{\Endo}[1][]{\mathrm{End}_{\raise1.5ex\hbox to.1em{}#1}}

\newcommand{\Hom}[1][]{\mathrm{Hom}_{\raise1.5ex\hbox to.1em{}#1}}

\newcommand{\RHom}[1][]{\mathrm{RHom}_{\raise1.5ex\hbox to.1em{}#1}}

\newcommand{\Ext}[2][]{\mathrm{Ext}_{\raise1.5ex\hbox to.1em{}#1}^{#2}}

\newcommand{\THom}[1][]{\mathrm{THom}_{\raise1.5ex\hbox to.1em{}#1}}

\newcommand{\Tens}[1][]{\mathbin{\otimes_{\raise1.5ex\hbox to-.1em{}#1}}}

\newcommand{\LTens}[1][]{\mathbin{\otimes_{\raise1.5ex\hbox to-.1em{}#1}^{L}}}

\newcommand{\Tor}[2][]{\mathrm{Tor}^{\raise1.5ex\hbox to.1em{}#1}_{#2}}

\newcommand{\shendo}[1][]{{\mathcal{E}nd}_{\raise1.5ex\hbox to.1em{}#1}}

\newcommand{\shaut}[1][]{{\mathcal{A}ut}_{\raise1.5ex\hbox to.1em{}#1}}

\renewcommand{\hom}[1][]{{\mathcal{H}om}_{\raise1.5ex\hbox to.1em{}#1}}

\newcommand{\rhom}[1][]{{R\mathcal{H}om}_{\raise1.5ex\hbox to.1em{}#1}}

\newcommand{\ext}[2][]{{\mathcal{E}xt}_{\raise1.5ex\hbox to.1em{}#1}^{#2}}

\newcommand{\thom}[1][]{{T\mathcal{H}om}_{\raise1.5ex\hbox to.1em{}#1}}

\newcommand{\tens}[1][]{\mathbin{\otimes_{\raise1.5ex\hbox to-.1em{}#1}}}

\newcommand{\ltens}[1][]{\mathbin{\otimes_{\raise1.5ex\hbox to-.1em{}#1}^{L}}}

\newcommand{\tor}[2][]{{\mathcal{T}or}^{\raise1.5ex\hbox to.1em{}#1}_{#2}}

\newcommand{\GHom}[1][]{\mathrm{GHom}_{\raise1.5ex\hbox to.1em{}#1}}

\newcommand{\GExt}[2][]{\mathrm{GExt}_{\raise1.5ex\hbox to.1em{}#1}^{#2}}

\newcommand{\FHom}[1][]{\mathrm{FHom}_{\raise1.5ex\hbox to.1em{}#1}}

\newcommand{\ghom}[1][]{{\mathcal{GH}om}_{\raise1.5ex\hbox to.1em{}#1}}

\newcommand{\gext}[2][]{{\mathcal{GE}xt}_{\raise1.5ex\hbox to.1em{}#1}^{#2}}

\newcommand{\fhom}[1][]{{\mathcal{FH}om}_{\raise1.5ex\hbox to.1em{}#1}}

\newcommand{\tenstop}[1][]{\mathbin{\hat{\otimes}_{\raise1.5ex\hbox to-.1em{}#1}}}

\newcommand{\homtop}[1][]{\mathcal{L}_{\raise1.5ex\hbox to.1em{}#1}}

\newcommand{\Homtop}[1][]{\mathrm{L}_{\raise1.5ex\hbox to.1em{}#1}}

\newcommand{\B}{\mathcal{B}}

\def\absdoim#1{\underline{#1}_*}
\def\reldoim[#1]#2{\underline{#2}_{|{#1}*}}
\def\doim{\@ifnextchar [{\reldoim}{\absdoim}}

\def\absdeim#1{\underline{#1}_*}
\def\reldeim[#1]#2{\underline{#2}_{|{#1}*}}
\def\deim{\@ifnextchar [{\reldeim}{\absdeim}}

\def\absdopb#1{\underline{#1}^{-1}}
\def\reldopb[#1]#2{\underline{#2}_{|{#1}}^{-1}}
\def\dopb{\@ifnextchar [{\reldopb}{\absdopb}}

\def\absboim#1{\underline{\underline{#1}}_*}
\def\relboim[#1]#2{\underline{\underline{#2}}_{|{#1}*}}
\def\boim{\@ifnextchar [{\relboim}{\absboim}}

\def\absbeim#1{\underline{\underline{#1}}_*}
\def\relbeim[#1]#2{\underline{\underline{#2}}_{|{#1}*}}
\def\beim{\@ifnextchar [{\relbeim}{\absbeim}}

\def\absbopb#1{\underline{\underline{#1}}^*}
\def\relbopb[#1]#2{\underline{\underline{#2}}_{|{#1}}^*}
\def\bopb{\@ifnextchar [{\relbopb}{\absbopb}}

\newcommand{\Aut}[1][]{\operatorname{Aut}_{\raise1.5ex\hbox to.1em{}#1}} 

\newcommand{\cat}{\mathsf}

\newcommand{\catHom}[1][]{\mathsf{1Hom}_{\raise1.5ex\hbox to.1em{}#1}}

\newcommand{\catEnd}[1][]{\mathsf{1End}_{\raise1.5ex\hbox to.1em{}#1}}
\newcommand{\catAut}[1][]{\mathsf{1Aut}_{\raise1.5ex\hbox to.1em{}#1}}
\newcommand{\catSh}{\mathsf{Sh}} 
\newcommand{\catLcSh}{\mathsf{LcSh}} 
\newcommand{\catCstSh}{\mathsf{CstSh}} 
\newcommand{\catCstS}{\mathsf{CstS}} 

\newcommand{\CatFct}[1][]{2{\mathbf{Fct}}_{\raise1.5ex\hbox to.1em{}#1}}

\newcommand{\stkEqui}[1][]{\mathfrak{Equi}_{\raise1.5ex\hbox to.1em{}#1}}

\begin{document}

\title[A pair of quasi-inverse functors for an extension of perverse sheaves]
{A pair of quasi-inverse functors for an extension of perverse sheaves}

\author[D.\ Dupont]{Delphine Dupont}
\address{University of Oxford\\Mathematical institute 24Ð29 St Giles'
Oxford
OX1 3LB\\}
\email{delphine@maths.ox.ac.uk}

\subjclass{  }
\date{}


\begin{abstract}
In \cite{McV}, the authors show that on a Thom-Mather space $X$ the category $\Pc erv_{X}$ of perverse sheaves is equivalent to the category $\Cc(F,G,T)$ whose objects are data of perverse sheaves on the complementary of the closed strata $S$, a local system on $S$ and some gluing data.  To show this equivalence of categories, they define a functor $\Cc$ going from the category $\Pc erv_{X}$ to the category $\Cc(F,G,T)$. This definition is based on the notion of perverse link. They do not define a quasi-inverse of this functor. moreover they have to consider first the case where $S$ is contractible and then they extend the equivalence to the topological case using the stack theory. In this paper we propose to consider what we call a perverse closed set which is a bit different from a perverse link in order to define a quasi-inverse to the functor $\Cc$. Moreover we treat directly the topological case without using stack theory. 
\end{abstract}

\maketitle

\section*{Introduction}
Let $(X, \Sigma)$ be a Thom-Mather space and $S$ a closed stratum. In \cite{McV}, the authors describe the category $\Pc erv_{\Sigma}$ of $\Sigma$-perverse sheaves in an elementary way. They give the first step to an inductive method. Hence they show the equivalence between the category $\Pc erv_{\Sigma}$ and a category whose objects are data of a perverse sheaf on $X_{0}=X\backslash S$, a locally constant sheaf on $S$ and some gluing data. Their first idea was to consider the triangle:
\[R\Gamma_{S}\Fc\rightarrow \Fc\rightarrow R\Gamma_{X_{0}}\Fc  \]
Indeed, a perverse sheaf $\Fc$ is isomorphic to the cone of the morphism:
\[ R\Gamma_{X_{0}}\Fc\longrightarrow R\Gamma_{S}\Fc[1]\]
and it is easy to show that $R\Gamma_{X_{0}}\Fc$ depends only of the restriction of $\Fc$ to $X_{0}$. More precisely, they show that this morphism can be defined up to isomorphism from the data  of $\Fc\mid_{X_{0}}$, a local system on $S$ and some gluig data.  Hence they introduce a category called zig-zag category, wich objects are this kind of data but this category is not equivalent to $\Pc erv_{X}$. They define a functor from $\Pc erv_{\Sigma}$ to the zig-zag category and they show that this functor is a bijection on the classes of objects and full, but not faithful. Actually, this comes from the fact that  even if  we can define functorially a morphism of complexes of sheaves from the data of an object of the zig-zag category, taking the cone of such a morphism is not functorial. 

Hence they have to introduce a new category $\Cc(F,G,T)$. It is defined from the data of two functors $F$ and $G$ and a natural transform $T$ depending of a choice of a set called perverse link and they show that this category is equivalent to the category $\Pc erv_{\Sigma}$. To show the equivalence they don't define a couple of quasi-inverse functors: first they show it in the case where $S$ is contractible using the zig-zag category, then they extend the equivalence to the topological case using the stack theory. 

The purpose of this paper is to give an equivalence of $\Pc erv_{\Sigma}$ with the category $\Cc(F', G',T')$ where the data $F'$, $G'$ and $T'$ are sightly different from the one in \cite{McV}, notably based on a different definition of a perverse closed set,  but for which the couple of quasi-inverse functors is explicit. One other advantage of this point of view is that we treat the topological case directly without using the stack theory. 

The point is that if $\Kc$ is a closed set with good properties and $\Ll$ is its complement in $X$, a perverse sheaf $\Fc$ is functorially isomorphic to the cone of the morphism:
\[R\Gamma_{\Ll}\Fc\longrightarrow R\Gamma_{\Kc}\Fc[1]\]
This comes from the lemma \ref{homologique} that can't be apply in the preceding case.
Moreover this morphism can be functorially defined up to isomorphism from the data of an object in $\Cc(F',G',T')$ using a  classical property of sheaves on stratified space. 

This paper is organized as follow.
In section $1$ we first recall the construction given by McPherson and Vilonen in \cite{McV} of an abelian category from the data of two abelian categories, two functors and a natural transform. Then we recall a classical gluing property of sheaves on a stratified space. These propositions are key fact in the main proof. 
Then in section $2$ we define what we call a perverse closed  set. It allows us to define two functors $F$ and $G$ and a natural transform $T$ from $F$ to $G$, in order to define the new category $\Cc(F,G,T)$. But even if this category is slightly different from the one of MacPherson and Vilonen, the object are still data of a perverse sheaf on $X_{0}$ a local system on $S$ and gluing data. The only change appear on the gluing data. Then using the gluing property of sheaves on a stratified space we define the two functors $\Cc: \Pc erv_{X}\rightarrow \Cc(F,G,T)$ and $\Pc: \Cc(F,G,T)\rightarrow \Pc erv_{X}$ and we show that they are quasi-inverse. 
\section{Preliminaries}
\subsection{The abelian category $\Cc(F,G,T)$}
Let us recall the definition and some properties of the category $\Cc(F,G,T)$ defined by MacPherson and Vilonen in \cite{McV}. \\
Let $\Ac$ and $\Bc$ be two categories, $F:\Ac\rightarrow \Bc$ and $G:\Ac\rightarrow \Bc$ two functors and $T:F\rightarrow G$ a natural transformation. 
\begin{Def}
We denote $\Cc(F,G,T)$ the category whose:
\begin{itemize}
\item objects are the families $(A,B,u,v)$ where A is an object of $\Ac$, $B$ is an object of $\Bc$ and $u:F(A)\rightarrow B$ and $v:B\rightarrow G(A)$ are morphisms of $\Bc$ such that the following diagram commutes:
\[\xymatrix{
F(A)\ar[rr]^{T_A}\ar[rd]_u & & G(A)\\
&B \ar[ru]_v
}\]
\item morphisms between two objects $(A,B,u,v)$ and $(A',B',u',v')$ are the data of a couple of morphisms $(a,b)$ where $a:A\rightarrow A'$ is a morphism of $\Ac$ and $b:B\rightarrow B'$ is a morphism of $\Bc$ such that the diagram commutes:
\[
\xymatrix @!0 @R=3pc @C=5pc {
     FA\ar[rr]^{T_{A}} \ar[rd]_u \ar[dd]_{Fa} &&  GA \ar[dd]^{Ga} \\
    & B \ar[ru]_{v} \ar[dd]_{b} \\
    FA'\ar[rr] |!{[ur];[dr]}\hole \ar[rd]_{u'} && GA' \\
    & B' \ar[ru]_{v'} }
\]
\end{itemize}
\end{Def}

\begin{prop}
If $\Ac$ and $\Bc$ are abelian, if $F$ is right exact and $G$ is left exact then $\Cc(F,G,T)$ is abelian.
\end{prop}

\subsection{Gluing of sheaves}
Now let us recall a classical property of sheaves on a space stratified by two strata.  In fact we give an answer to the question: ``what supplementary data are needed to recover a sheaf from the data of its restrictions to each stratum ?''. More precisely we show the equivalence between the category of sheaves and a category whose objects are  data of a sheaf on each stratum plus a morphism of sheaves. \\

Let $X$ be a topological space, $F$ a closed set of $X$ and $U$ the complement of $F$ in $X$. One denotes $i$ and $j$ the embeddings:
\[ i: F\hookrightarrow X, ~~~j:U\hookrightarrow X\]
\begin{Def}
Let $S_F$ be the category whose 
\begin{itemize}
\item[$\bullet$] objects are the data of a triple $(\Fc_F, \Fc_U, f)$ where $\Fc_F$ and $\Fc_U$ are sheaves on respectively $F$ and $U$ and $f$ is a morphism of sheaves:
\[ f:\Fc_F\longrightarrow i^{-1}j_*\Fc_U\]
\item[$\bullet$] morphisms between two objects $(\Fc_F, \Fc_U, f)$ and $(\Gc_F, \Gc_U, g)$ are couples $(\phi_F, \phi_U)$ of morphisms of sheaves $\phi_F:\Fc_F\rightarrow \Gc_F$, $\phi_U:\Fc_U\rightarrow \Gc_U$ such that the diagram commutes: 
\[\xymatrix{
\Fc_F  \ar[r]^-{f} \ar[d]_{\phi_F} & i^{-1}j_*\Fc_U \ar[d]^{i^{-1}j_*\phi_U}\\
\Gc_F \ar[r]_-{g} & i^{-1}j_*\Gc_U
}\]
\end{itemize}
\end{Def}
\begin{prop}
The category $\catSh_X$ of sheaves on $X$ is equivalent to the category $\cat S_F$.
\end{prop}
\begin{proof}
We just define a couple of functors, the fact that they are quasi-inverse is left to the reader. \\
We denote $R_F$ the ``restriction'' functor going from $\catSh_X$ to $\cat S_F$ defined by:
\[  \begin{array}{cccc}
R_F: &\catSh_X & \longrightarrow & \cat S_F\\
& \Fc & \longmapsto & (\Fc\mid_F,\Fc\mid_U, \eta )\\
& (\phi: \Fc\rightarrow \Gc) & \longmapsto & (\phi\mid_F, \phi\mid_U)
\end{array}  \]
where $\eta$ is the restriction of the natural morphism given by the adjunction:
\[ \eta: i^{-1}\Fc\longrightarrow i^{-1}j_*j^{-1}\Fc. \]
Let us denote $G_F$ the ``gluing" functor going from $\cat S_F$ to $\catSh_X$ that associate to a family $(\Fc_F, \Fc_U, f)$ the pull-back of the diagram:
\[\xymatrix{
& i_*\Fc_F \ar[d]^{i_*f}\\
j_*\Fc_U \ar[r]_-{\eta}&  i_*i^{-1}j_*\Fc_U
}\]
Where $\eta$ is the natural morphism given by the adjunction between the functors $i_*$ and $i^{-1}$. \\
If $(\phi_F, \phi_U)$ is a morphism from $(\Fc_F, \Fc_U, f)$ to $(\Gc_F, \Gc_U, g)$ the commutations conditions and the universal property assure the existence of a morphism from $G_F\big((\Fc_F, \Fc_U, f)\big)$ to $G_F\big((\Gc_F, \Gc_U, g)\big)$. This morphism is the image of $(\phi_F, \phi_U)$ by $G_F$. 
\end{proof}
Now let us suppose that $U$ is an open dense set in $X$. This implies that the couple $(F,U)$ is a stratification of $X$. \\
Let us recall that a constructible sheaf relatively to the stratification $(F,U)$ is a sheaf which the restrictions to $U$ and $F$ are locally constant sheaves. One denotes $\catCstSh_F$ the full sub-category of $\catSh_{X}$ whose objects are constructible sheaves. In the same way, let us denote $\catCstS_F$ the full-subcategory of $\cat S_F$ whose objects are the families with locally constant sheaves on $F$ and on $U$. The restriction of  $R_F$ to $\catCstSh_F$ and the restriction of $G_F$ to $\catCstS_F$ are a couple of quasi-inverse functors. \\

\section{Extension of perverse sheaves}
 
 Let $(X,\Sigma)$ be a Thom Mather space with only one closed stratum $S$ of  minimal dimension $d$, $X_0$ its complement $X\backslash S$ and $i$ and $j$ the embeddings:
 \[j:X_0\hookrightarrow X;~~~i:S\hookrightarrow X \]
 We denote $\Sigma_{0}$ the stratification of $X_{0}$ given by the stratification $\Sigma$ minus the stratum $S$.\\
 
 In this section we use the preceding construction with $\Ac$ the category of $\Sigma_{0}$-perverse sheaves on $X_{0}$ and $\Bc$ the category of locally constant sheaves on $S$. Hence objects of the category $\Cc(F,G,T)$ will be a quadruple $(\Fc, \Ll, u,v)$ where $\Fc$ is a perverse sheaf on $X_{0}$, $\Ll$ a locally constant sheaf on $S$ and $u$ and $v$ morphisms of sheaves that satisfy a commutation condition. The morphisms $u$ and $v$ formed the gluing data. 
 
 In a first time we define the functors $F$, $G$ and the natural transform $T$ to defined $\Cc(F,G,T)$ these definitions are based on the notion of perverse closed set. Then we define a couple of quasi-inverse $\Cc$ and $\Pc$ between the categories $\Pc erv_{\Sigma}$ and $\Cc(F,G,T)$.

 \begin{Def}
 Let $\Kc$ be a closed set of $X$ and $\Ll$ be its complement in $X$, it is a perverse closed set if for any perverse sheaf $\Fc \in \Pc erv_{\Sigma_0}$ we have:
 \begin{itemize}
 \item $\forall k<-d, R^k\Gamma_{\Kc}Rj_*\Fc=0$,
 \item  $\forall k\geq-d, R^k\Gamma_{\Ll}Rj_*\Fc=0$,
 \end{itemize}
 \end{Def}
 
 \begin{prop}
Perverse closed set exists.
 \end{prop}
 \begin{bf}
 Examples
 \end{bf}
 \begin{itemize}
 \item Let us consider the topological space $\CC^{n}$ stratified by the normal crossing stratification, according to Galligo, Granger and Maisonobe \cite{GGM} the  set $\RR^{n}$ is a perverse closed set. 
 \item Let $X=\{(u,v,w)\in \CC^{3}\mid uw=v^{2}\}$ be the affine toric variety. The $2$-torus act on $X$:
 \[\begin{array}{cccc}
 (\CC^{*})^{2} \times X : &\longrightarrow & X\\
 ( (t_{1}, t_{2}), (u,v,w))& \longmapsto & (t_{1}u, t_{1}t_{2}v, t_{1}t_{2}^{2}w)
 \end{array}\]
 The orbits of this action provide a Whitney stratification. Then the set $X\cap \{(u,v,w)\in \CC^{3}\mid u\in \RR^{+}, w\in \RR^{+}\}$ is a perverse closed set for this stratification.  Let us show it: 
 \end{itemize}
\begin{remark}\label{remarque ferme pervers}
If $\Kc$ is a perverse closed set and $\Ll$ is its complement in $X$, then $\Ll\cap S=\emptyset$. Indeed let us consider $\Kc$ and $\Ll$ a closed set an its complement in $X$ and the perverse sheaf $\CC_{S}[d]$ where $\CC_{S}$ is the constant sheaf supported by $S$.  If $x\in \Ll\cap S$ we have:
\[(R^{-d}\Gamma_{\Ll}\CC_{S}[d])_{x}\simeq(\Gamma_{\Ll}\CC_{S})_{x}=\CC \]
and $\Kc$ is not a perverse closed set. 
\end{remark}

Now let us fix a perverse closed set $\Kc$, $\Ll$ will denote its complement in $X$, let $j_{\Ll}$ and $i_{\Kc}$ the embeddings of respectively $\Ll$ and $\Kc$ in $X$:
\[ j_{\Ll}:\Ll\hookrightarrow X; ~~~~i_{\Kc}:\Kc \hookrightarrow X \]
\begin{Def}
Let us denote respectively $F$ and $G$ the functors defined by:
\[\begin{array}{ccccc}
F:& \Pc erv_{\Sigma_0} &\longrightarrow & \catLcSh_{S}\\
& \Fc &\longmapsto & (R^{-d-1}\Gamma_{\Ll}Rj_{*}\Fc)\mid_{S}\\
G:& \Pc erv_{\Sigma_0} &\longrightarrow & \catLcSh_{S}\\
& \Fc &\longmapsto & (R^{-d}\Gamma_{\Kc}Rj_{*}\Fc)\mid_{S}\\
\end{array}\]
We denote $T_{\Fc}$ the natural transformation:
\[ T_{\Fc}:  (R^{-d-1}\Gamma_{\Ll}Rj_{*}\Fc)\mid_{S}\longrightarrow  (R^{-d}\Gamma_{\Kc}Rj_{*}\Fc)\mid_{S}\]
\end{Def}
The aim of this paper is to show that the category $\Cc(F,G,T)$ is equivalent to the category $\Pc erv_{\Sigma}$, defining a couple of quasi-inverse functors. Let us remark that as the functors $F$ and $G$ are different, the category $\Cc(F,G,T)$ is different from the one given by MacPherson and Vilonen, but the difference is only on the gluing data.  

In order to define these functors, we need to prove this lemma on perverse closed set.
\begin{lem}
Let $\Kc$ be a closed set and $\Ll$ its complement in $X$.
\begin{itemize}\label{ferme pervers}
\item[(i)] The closed set $\Kc$ is a perverse closed set if and only if for all perverse sheaf $\Fc \in \Pc erv_{\Sigma}$ we have:
\begin{itemize}
\item $\forall k<-d, R^{k}\Gamma_{\Kc}\Fc=0$
\item $\forall k\geq -d, R^{k}\Gamma_{\Ll}\Fc=0$
\end{itemize}
\item[(ii)] If $\Kc$ is a perverse closed set then for any perverse sheaves $\Fc\in \Pc erv_{\Sigma}$ we have:
\[R\Gamma_{\Kc}\Fc\simeq R^{-d}\Gamma_{\Kc}\Fc[d] \]
\end{itemize}
\end{lem}
\begin{proof}
The first implication of $(i)$ and $(ii)$ can both be shown using the perversity conditions and the triangle:
\[R\Gamma_{\Kc}\Fc\rightarrow \Fc\rightarrow R\Gamma_{\Ll}\Fc\]
The second implication of $(i)$ can be deduced from the fact that for each perverse sheaf $\Fc$ on $X_{0}$ there exists a perverse sheaf on $X$ such that its restriction to $X_{0}$ is isomorphic to $\Fc$.
\end{proof}
Let us define the functor $\Cc$ from the category $\Pc erv_{\Sigma}$ to the category $\Cc(F,G,T)$. This definition is similar to the one given by MacPherson and Vilonen in \cite{McV}, the only change comes from the different definition of the functors $F$ and $G$. \\
Let $\Fc$ be a perverse sheaf on $X$. It follows from the remark \ref{remarque ferme pervers} that the natural morphism:
\[R\Gamma_{\Ll}\Fc\longrightarrow R\Gamma_{\Ll}Rj_{*}j^{-1}\Fc \]
is an isomorphism. Hence if we denote $u$ the composition of the inverse of this isomorphism with the natural morphism:
\[ i^{-1}R^{-d-1}\Gamma_{\Ll}\Fc\longrightarrow i^{-1}R^{-d}\Gamma_{\Kc}\Fc\]
and $v$ the natural morphism:
\[ i^{-1}R^{-d-1}\Gamma_{\Kc}\Fc\longrightarrow i^{-1}R^{-d}\Gamma_{\Kc}Rj_{*}j^{-1}\Fc, \]
then the diagram commutes: 
\[\xymatrix{
F(i^{-1}\Fc) \ar[rr]^{T} \ar[dr]_{u}& & G(i^{-1}\Fc)\\
&  i^{-1}R^{-d-1}\Gamma_{\Kc}\Fc \ar[ur]_{v}
}\]
Moreover, all these operations are functorial.
\begin{Def}
We denote $\Cc$ the functor going from $\Pc erv_{\Sigma}$ to $\Cc(F,G,T)$ defined, with the above notation, by:
\[\begin{array}{cccc}
\Cc: & \Pc erv_{\Sigma} & \longrightarrow & \Cc(F,G,T)\\
&\Fc& \longmapsto &(j^{-1}\Fc,i^{-1}R^{-d}\Gamma_{\Kc}\Fc,u,v)
\end{array}\]
\end{Def}
We are now going to define the functor $\Pc:\Cc(F,G,T) \rightarrow \Pc erv_{\Sigma}$ quasi-inverse of $\Cc$. But let us first give a general idea of the construction. Let $(\Gc, \Ll, u,v)$ be an object of $\Cc(F,G,T)$. Let us recall that $\Gc$ is a perverse sheaf on $X_{0}$, $\Ll$ a locally constant sheaf on $S$ and $u$ and $v$ are morphisms of sheaves such that the following diagram commutes:
\[\xymatrix{
i^{-1}R^{-d-1}\Gamma_{\Ll}Rj_{*}\Gc \ar[rr]^{T} \ar[rd]_{u} &&i^{-1}R^{-d}\Gamma_{\Kc}Rj_{*}\Gc\\
& \Ll\ar[ur]_{v}
}\]
Starting from this data and using the ``gluing'' functor $G_{S}$ defined in the preceding section we define a morphism of sheaves, denoted $\phi$ from $R^{-d-1}\Gamma_{\Ll}Rj_{*}\Gc$ to a sheaf $\Bc$. Then as $R\Gamma_{\Ll}Rj_{*}\Gc$ is concentrated in degree less than $-d-1$ we have:
\[ Hom_{\catSh_{X}}(R^{-d-1}\Gamma_{\Ll}Rj_{*}\Gc, \Bc)\simeq Ext^{1}(R\Gamma_{\Ll}Rj_{*}\Gc, \Bc[-d]) \]
Hence the morphism $\phi$ defined a morphism $\Phi$ of complexes of sheaves. 
The perversity conditions of the closed set $\Kc$ and of the perverse sheaf $\Gc$ allow us to apply the following lemma to show that the cone of such a morphism is unique.
\begin{lem}\label{homologique}
Let $K$ be a triangulated category. Let us consider the commutative diagram:
\[\xymatrix{ A\ar[r]^{u} \ar@{.>}[d]_{\alpha} & B\ar[r]^{v} \ar[d]_{\beta} & C\ar[r]^{w} \ar[d]_{\gamma}  &A[1]  \ar@{.>}[d]_{\alpha[1]}\\
A' \ar[r]_{u'} & B'\ar[r]_{v'} & C'\ar[r]_{w'}& A'[1]}\]
then there exists a morphism $\alpha$ such that the diagram commutes. If moreover we have:
\[ Hom_{K}(B, C'[-1])=Hom_{K}(C,B')=0 \]
then the morphism $\alpha$ is unique. 
\end{lem}
\begin{proof}
See \cite{Mai}.
\end{proof}
Hence we define the image of $(\Gc, \Ll, u,v)$ by $\Pc$  to be the cone of $\Phi$. Moreover this lemma shows that taking the cone of such a morphism is functorial.

Now let us define properly the functor $\Pc$. Let  $\gimel:=(\Gc, \Ll, u,v)$ be an object of $\Cc(F,G,T)$. The family $(\Ll, j^{-1}R^{-d}\Gamma_{\Kc}Rj_{*}\Gc, \tilde{v})$ where $\tilde{v}$ is the composition of the morphism $v$ with the adjonction morphism $\eta$:
\[ \eta: i^{-1}R^{-d}\Gamma_{\Kc}Rj_{*}\Gc\longrightarrow i^{-1}j_{*}j^{-1}R^{-d}\Gamma_{\Kc}Rj_{*}\Gc\]
is an object of $\cat S_{S}$ in the same way the couple $(u,\delta)$ where $\delta$ comes from the natural triangle is a morphism in $\cat S_{S}$ from $(i^{-1}R^{-d-1}\Gamma_{\Ll}Rj_{*}\Gc,$ $j^{-1}R^{-d-1}\Gamma_{\Ll}j_{*}\Gc)$ to $(\Ll, j^{-1}R^{-d}\Gamma_{\Kc}Rj_{*}\Gc, \tilde{v})$. Let us denote $\Bc_{\gimel}$ the sheaf $G_{S}\big((\Ll,  j^{-1}R^{-d}\Gamma_{\Kc}Rj_{*}\Gc, \tilde{v})\big)$ and $\phi_{\gimel}$ the morphism $G_{S}\big((u, \delta)\big)$:
\[\phi_{\gimel}:=G_{S}\big((u, \delta)\big): R^{-d-1}\Gamma_{\Ll}Rj_{*}\Gc\longrightarrow \Bc_{\gimel}\] 
As $R\Gamma_{\Ll}Rj_{*}\Gc$ is concentrated in degree smaller than $-d-1$ the two complexes $R\Gamma_{\Ll}Rj_{*}\Gc$ and $\tau^{\leq-d-1}R\Gamma_{\Ll}Rj_{*}\Gc$ are isomorphic. Let us denote $\Phi_{\gimel}$ the morphism of complexes defined by the composition of the natural morphisms and $\phi$ shifted by $-d-1$:
\[R\Gamma_{\Ll}Rj_{*}\Gc\buildrel\sim\over\rightarrow \tau^{\leq-d-1}R\Gamma_{\Ll}Rj_{*}\Gc\rightarrow R^{-d-1}\Gamma_{\Ll}Rj_{*}\Gc[-d-1]\buildrel\phi_{\gimel}[d+1]\over\rightarrow\Bc_{\gimel}[-d]\]
If $(g,l)$ is a morphism in $\cat S_{S}$ from $\gimel=(\Gc, \Ll, m,n)$ to $\gimel'=(\Gc', \Ll', m',n')$ we can associate using the gluing functor $G_{S}$ a morphism $b_{(g,l)}$ from $\Bc_{\gimel}$ to $\Bc_{\gimel'}$ such that the diagram commutes:
 \[\xymatrix{
 R\Gamma_{\Ll}Rj_{*}\Gc\ar[r]^{\Phi_{\gimel}}\ar[d]_{R\Gamma_{\Ll}Rj_{*}g} & \Bc_{\gimel}\ar[d]^{b_{(g,l)}}\\
 R\Gamma_{\Ll}Rj_{*}\Gc' \ar[r]_{\gimel} & \Bc_{\gimel'}
 }\]
 Moreover we have $b_{(id,id)}=Id$ and if $(g,l)$ and $(g',l')$ are two composable morphisms then we have $b_{(g,l)}\circ b_{(g',l')}=b_{(g\circ g',l\circ l')}$
The following lemma is the central point of the definition of the functor $\Pc$.
 \begin{lem}\label{unicitŽ cone}
 The cone of $\Phi_{\gimel}$ is unique and moreover the application  which associate this cone to an object of $\Cc(F,G,T)$ is functorial.  \end{lem}
 \begin{proof}
 First, let us fix, for every object $\gimel$ of $\Cc(F,G,T)$ a cone $\Pc(\gimel)$ of $\Phi_{\gimel}$.
 To show the lemma it suffice to show that if $(g, l)$ is a morphism in $\cat S_{S}$ from $\gimel=(\Gc, \Ll, u,v)$ to $\gimel'=(\Gc', \Ll', u',v')$, then there exist a unique morphism such that the diagram commutes:
 \begin{equation}\label{diagram morphism}
 \xymatrix{
 R\Gamma_{\Ll}Rj_{*}\Gc \ar[r] \ar[d]_{R\Gamma_{\Ll}Rj_{*}g} & \Bc_{\gimel}[d] \ar[r] \ar[d]_{b_{(g,l)}} & \Pc (\gimel)\ar@{.>}[d]\\
 R\Gamma_{\Ll}Rj_{*}\Gc' \ar[r]  & \Bc_{\gimel'}[d] \ar[r]  & \Pc (\gimel') 
 }\end{equation}
By definition of a triangulated category we know that such a morphism exists then let us show that it is unique. Because of the degree of the complexes we have:
\[ Hom(\Bc[d], R\Gamma_{\Ll}Rj_{*}\Gc')=0\]
Moreover, as $Rj_{\Ll*}$ is a left adjoint of $j^{-1}_{\Ll}$ we have:
\[Hom(R\Gamma_{\Ll}Rj_{*}\Gc, \Bc'[d])=Hom(j_{\Ll}^{-1}Rj_{*}\Gc, j_{\Ll}^{-1}\Bc'[d])\]
but $\Bc'$ is supporting by $\Kc$, hence the complex is null. Then the lemma \ref{homologique} applies. We denote $\Pc\big((g,l)\big)$ this morphism.
 \end{proof}
 \begin{Def}
Using the above notation, we denote $\Pc$ the functor going from the category $\Cc(F,G,T)$ to the derived category of sheaves with constructible cohomology  which associate:
\begin{itemize}
\item  to an object $\gimel$ of $\Cc(F,G,T)$ the cone of the morphism $\Phi_{\gimel}$,
\item to a morphism $(g,l):\gimel\rightarrow \gimel'$ in $\Cc(F,G,T)$ the unique morphism such that the diagram \ref{diagram morphism} commutes.
\end{itemize}
 \end{Def}
 The following lemma is going to be very useful.
 \begin{lem}\label{pervers}
 Let $\gimel=(\Gc, \Ll, u,v)$ be an object of $\Cc(F,G,T)$, then the restriction of $\Pc(\gimel)$ to $X_{0}$ is naturally isomorphic to $\Gc$.
 \end{lem}
 \begin{proof}
 First the complex $j^{-1}\Pc(\gimel)$ is naturally isomorphic to the cone of the morphism $j^{-1}\Phi_{\gimel}$. Then let us remark that because of the perversity condition of the closed set $\Kc$ the complex $j^{-1}Rj_{*}\Gc$ is the unique cone of the morphism:
 \[ j^{-1}R\Gamma_{\Ll}Rj_{*}\Gc \buildrel[+1]\over\longrightarrow j^{-1}R\Gamma_{\Kc}Rj_{*}\Gc[1]\]
 In view of the definition of $\Bc_{\gimel}$ and as the functors $R_{S}$ and $G_{S}$ are quasi-inverse, there exists a natural isomorphism between $j^{-1}\Bc_{\gimel}$ and $j^{-1}R\Gamma_{\Kc}Rj_{*}\Gc$ such that the following diagram commutes:
 \[\xymatrix{
 j^{-1}R\Gamma_{\Ll}Rj_{*}\Gc \ar[rrr]^-{j^{-1}\Phi} \ar@{=}[d]&&&j^{-1}\Bc_{\gimel}[d] \eq[d] \\
j^{-1}R\Gamma_{\Ll}Rj_{*}\Gc \ar[rrr]_{[1]} &&& j^{-1}R\Gamma_{\Kc}Rj_{*}\Gc [1]
 }\]
As in the proof of lemma  \ref{unicitŽ cone}, we show that the lemma \ref{homologique} applies. Hence there exists a natural isomorphism between $j^{-1}\Pc (\gimel)$ and $j^{-1}Rj_{*}\Gc$. It remains to say that $j^{-1}Rj_{*}\Gc$ is naturally isomorphic to $\Gc$.
 \end{proof}
\begin{prop}
The functor $\Pc$ takes values in the category $\Pc erv_{\Sigma}$.
\end{prop}
\begin{proof}
Let $\gimel=(\Gc, \Ll, u,v)$ be an object of $\Cc(F,G,T)$. The proof is based on the following result showed in \cite{BBD}.
\begin{prop}
Let $Y$ be a closed submanifold of $X$ of dimension $d$ and $U$ its complement in $X$. Let us denote $i$ and $j$ the embeddings $i:Y\hookrightarrow X$ and $j:U\hookrightarrow X$. Let $\Fc$ be a complex with constructible cohomology, then $\Fc$ is a perverse sheaves if and only if:
\begin{itemize}
\item $j^{-1}\Fc$ is a perverse sheaves,
\item $i^{-1}\Fc$ is concentrated in degrees lesser than $-d$,
\item $(R\Gamma_{Y}\Fc)\mid_{Y}$ is concentrated in degrees bigger than $-d$.
\end{itemize}
\end{prop}
We have already seen (lemma \ref{pervers}) that the complex $j^{-1}\Pc (\gimel)$ is perverse.\\
Let us consider the complex $R\Gamma_{S}\Pc(\gimel)$. It is isomorphic to the cone of the morphism $R\Gamma_{S}\Phi_{\gimel}$ defined before:
\[R\Gamma_{S}\Phi_{\gimel}:R\Gamma_{S}R\Gamma_{\Ll}Rj_{*}\Pc(\gimel)\longrightarrow R\Gamma_{S}\Bc[d]\]
But according to remark \ref{remarque ferme pervers}, $S\cap \Ll=\emptyset$, the first complex is null. This shows that $R\Gamma_{S}\Pc(\gimel)$ is concentrated in degrees bigger than $-d$. \\ 
\end{proof}

\begin{theorem}
The functors $\Pc$ and $\Cc$ are quasi-inverse. 
\end{theorem}
 \begin{proof}
 The proof of this theorem is essentially based on the lemma \ref{homologique} and on the fact that the functors $R_{S}$ and $G_{S}$ are quasi-inverse.\\
 Let us first prove that the functor $\Pc\Cc$ is isomorphic to the identity.
 Let $\gimel=(\Gc, \Ll, u,v)$ be an object of $\Cc(F,G,T)$. We denote $\Hc$ the perverse sheaves $\Pc(\gimel)=\Hc$ and $\gimel'=(\Hc\mid_{X_{0}}, R^{-d}\Gamma_{\Kc}\Hc\mid_{S}, u',v')$ the quadruple $\Cc\Pc(\gimel)=\gimel'$. We have already seen (lemma \ref{pervers}) that $j^{-1}\Hc$ is naturally isomorphic to $\Gc$. Let us consider the sheaf $(R^{-d}\Gamma_{\Kc}\Hc)\mid_{S}$. The complex $R\Gamma_{\Kc}\Hc$ is isomorphic to the cone of the morphism:
\[ R\Gamma_{\Kc}\Phi_{\gimel}: R\Gamma_{\Kc}R\Gamma_{\Ll}Rj_{*}\Gc\longrightarrow R\Gamma_{\Kc}\Bc_{\gimel}[d]\]
 But as $\Ll\cap \Kc=\emptyset$ and as $\Bc_{\gimel}$ is supported by $\Kc$ then $R\Gamma_{\Kc}\Hc$ is naturally  isomorphic to $\Bc_{\gimel}[d]$. Hence we have:
 \[(R^{-d}\Gamma_{\Kc}\Hc)\mid_{\Kc}\simeq\Bc_{\gimel}\mid_{\Kc}\]
 Hence by definition of $\B_{\gimel}$ and as $R_{S}$ and $G_{S}$ are quasi-inverse, the sheaves $(R^{-d}\Gamma_{\Kc}\Hc)\mid_{\Kc}$ and $\Ll$ are naturally isomorphic. Moreover, these isomorphisms commutes with the morphisms of sheaves $u'$ and $v'$.

 Now let us show that $\Cc\Pc$ is isomorphic to the identity. Let $\Fc$ be a perverse sheaf on $X$. 
 Let us recall the notations, the sheaf $\Bc_{\Cc(\Fc)}$ is the image by $G_{S}$ of the triplet $\big((R^{-d}\Gamma_{\Kc}Rj_{*}j^{-1}\Fc)\mid_{X_{0}},(R^{-d}\Gamma_{\Kc}\Fc)\mid_{S}, \tilde{v}\big)$,    
 \[\Bc_{\Cc(\Fc)}=G_{S}\Big(\big((R^{-d}\Gamma_{\Kc}Rj_{*}j^{-1}\Fc)\mid_{X_{0}},(R^{-d}\Gamma_{\Kc}\Fc)\mid_{S}, \tilde{v}\big)\Big)\]
 where $\tilde{v}$ is the composition of the morphisms:
 \[i^{-1}R^{-d}\Gamma_{\Kc}\Fc\rightarrow i^{-1}R^{-d}\Gamma_{\Kc}Rj_{*}j^{-1}\Fc\rightarrow i^{-1}j_{*}j^{-1}R^{-d}\Gamma_{\Kc}\Fc\]
The complex $\Fc$ is the unique cone of the morphism:
 \[ R\Gamma_{\Ll}\Fc\buildrel[+1]\over \longrightarrow R\Gamma_{\Kc}\Fc[1]\] 
 Indeed, the perversity conditions of the closed set $\Kc$ and of the complex $\Fc$ allow us to apply the lemma \ref{homologique}. Hence to show that $\Cc\Pc$ is isomorphic to the identity it suffices to define two natural isomorphisms such that the following diagram commutes:
 \[\xymatrix{
 R^{-d-1}\Gamma_{\Ll}Rj_{*}j^{-1}\Fc \eq[d]\ar[rrr]^{\phi_{\Cc(\Fc)}}&&& \Bc_{\Cc(\Fc)}\eq[d]\\
 R^{-d-1}\Gamma_{\Ll}\Fc \ar[rrr]_{[+1]}&&& R^{-d}\Gamma_{\Kc}\Fc
 }\]
 Let us denote $\delta$ the natural morphism:
 \[\delta: R^{-d-1}\Gamma_{\Ll}\Fc\longrightarrow R^{-d-1}\Gamma_{\Ll}Rj_{*}j^{-1}\Fc\]
and $\gamma$ be the natural morphism: 
\[ \gamma: R^{-d}\Gamma_{\Kc}\Fc\longrightarrow R^{-d}\Gamma_{\Kc}Rj_{*}j^{-1}\Fc\]
The couple $(i^{-1}\gamma, Id)$ is an isomorphism in $\cat S_{S}$ from $R_{S}( R^{-d}\Gamma_{\Kc}\Fc)$ to $\big((R^{-d}\Gamma_{\Kc}Rj_{*}j^{-1}\Fc)\mid_{X_{0}},(R^{-d}\Gamma_{\Kc}\Fc)\mid_{S}, \tilde{v}\big)$. Moreover as the suitable diagram of morphisms of sheaves on $X_{0}$ and on $S$ commutes and as $R_{S}$ and $G_{S}$ are a couple of quasi-inverse, we obtain that the following diagram commutes:
\[\xymatrix{
R^{-d-1}\Gamma_{\Ll}Rj_{*}j^{-1}\Fc \eq[d]_{\delta}\ar[rrr]^-{\phi_{\Cc(\Fc)}}&&& \Bc_{\Cc(\Fc)} \eq[d]^-{G_{S}((i^{-1}\gamma,Id))}\\
 G_{S}R_{S}(R^{-d-1}\Gamma_{\Ll}\Fc) \ar[rrr]_{G_{S}R_{S}([+1])}&&&G_{S}R_{S}( R^{-d}\Gamma_{\Kc}\Fc)
 }\]
 Then composing this isomorphisms with the one given by the quasi-inversibility of $R_{S}$ and $G_{S}$, we obtain the looked isomorphisms.

\end{proof}
\bibliographystyle{alpha}
\bibliography{biblio}

\end{document}